\documentclass[letterpaper, 10 pt, conference]{ieeeconf}
\pdfoutput=1
\IEEEoverridecommandlockouts                                                                                   
\usepackage[english]{babel}
\usepackage{graphicx}
\graphicspath{{Figures/}{../Figures/}{../../Figures}}
\usepackage{amsmath}
\usepackage{mathtools}
\mathtoolsset{showonlyrefs}
\usepackage{amssymb}
\usepackage[mathcal]{eucal}
\usepackage[T1]{fontenc}
\usepackage{cite}
\usepackage{hyperref}
\usepackage{tikz}
\usetikzlibrary{matrix,arrows,calc,automata}
\usepackage{pgfplots}
\pgfplotsset{compat=newest}
\usepackage{thmtools}
\usepackage{enumerate} 
\hypersetup{colorlinks, breaklinks, linkcolor=black,citecolor=black,filecolor=black,urlcolor=black}

\newtheorem{definition}{Definition}
\newtheorem{theorem}{Theorem}
\newtheorem{example}{Example}
\renewcommand\thmcontinues[1]{Continued}
\newtheorem{remark}{Remark}
\newtheorem{proposition}{Proposition}

\title{\LARGE \bf Data-Driven Analysis of Mass-Action Kinetics}

\author{Camilo Garcia-Tenorio, Duvan Tellez-Castro\thanks{Camilo Garcia-Tenorio and Duvan Tellez-Castro are supported by Colciencias-Doctorado Nacional-647/2015 \& 727/2016}, Eduardo Mojica-Nava, and Jorge Sofrony\thanks{Departamento de Ingenier\'ia, Universidad Nacional de Colombia, Bogot\'a, Colombia 111321, \texttt{\{cagarciate, datellezc, eamojican, jsofronye\}@unal.edu.co}}}
\begin{document}
\selectlanguage{english}
\maketitle
\begin{abstract}
The physical interconnection of spatial distributed biochemical systems has some advantages when dealing with large-scale problems that require separated agents to be controlled locally but with an overall objective. The analysis and control of such systems becomes a difficult task, because of the nonlinear nature of the dynamics of the individual subsystems, and the added complexity due to the interconnection. Therefore, analysis tools from a data-driven perspective are employed in contrast with the analytical classical way that becomes intractable once the subsystems grow in size or complexity.
\end{abstract} 

\section{Introduction} 
	\label{sec:introduction}

	The problem of wastewater treatment is a current issue in the path of sustainable intelligent cities. Modern cities can plan their growth and development centered on these perspectives, but older cities whose development was not sustainable centered cannot accommodate the large areas a centralized wastewater treatment plant occupies. Along this trend, city expansion could be focused on decentralized wastewater treatment in favor of better use of space and the avoidance of large facilities \cite{Larsen2016}. There are some challenges to be addressed for the decentralized management of water. One of them is in the dimensionality of the problem due to interconnection; the whole system has to guarantee the collaboration of subsystems to get a handle on an adequate effluent. On the other hand, the whole interconnected system in the wastewater treatment allows for its resilience in terms of collaboration between individual subsystems when one or several are operating under nonideal conditions.

	Biochemical processes for the wastewater treatment process such as the activated sludge treatment can be modeled by {\it mass action kinetics} (MAK). This representation, based on the stoichiometric network of biochemical processes give a mathematically rich polynomial structure based on concentrations \cite{Chellaboina2009a}. There are plenty of analysis tools to handle on such structure \cite{Haddad2010,Bernstein1999}. The problem arises when there is uncertainty in the reaction constants \cite{Cukier1973}, when there is uncertatinty, the general behavior of the system cannot be analyzed quantitavely in an exact manner, only qualitative characteristics are present. Also, there are problems  when an interconnection component is present on the system, that is a system of systems. This interconnection component induces a rise in the number of variables and can induce complex behavior. When this happens, classical tools have to be complemented with theories handling on large-scale systems and their behavior in order to analize the overall behavior and performance of the system. 

	The objective of this paper is to show how classical techniques such as Lyapunov, or passivity (Energy-based) analysis can be coupled with operator theory (Koopman operator) in order to have a measure of the individual and overall behavior and performance. There are many tools for the analysis of polynomial or algebraic systems \cite{Haddad2008,Haddad2010,Khalil2015}. The problem with these is that they can not get a handle on complexity; i.e., when the dimensionality or the order of the system grows, specifically in the case of complexity for polynomial systems \cite{Ahmadi2011,Majumdar2014,Anderson2010,giesl2015review}.
 	Traditional methods for the non-existence of periodic orbits such as Bendixon criteria or comparison functions need much effort for its implementation in high dimensions \cite{budivsic2012applied,skar1981nonexistence}. Another problem is that usually with the classic techniques (e.g., Lyapunov, input-output stability) the way to develop a global stability analysis, is checking if there are other points of equilibria or limit cycles \cite{cuesta1999stability}. But even the first step of finding all the equilibria points for the analysis of nonlinear systems is a hard problem. Therefore,  complex dynamic systems require tools for their analysis that transcend beyond the current knowledge in control theory \cite{dimirovski2016complex}. The escalating behavior in the number of equilibria and the characteristic of the equilibria do not allow for a simple and straightforward approach. Still, the overall performance has to be guaranteed under interconnection. Therefore, the approach proposed in this paper is to use operator based algorithms coupled with classical analysis tools. This allows us to have a better measure of the performance of the system; the characterization of attraction regions for this particular case when we only have data snapshots of the system and a general or qualitative picture of the process. Although the analysis tools are intended for the interconnection of systems, the proposed approach is exposed for the analysis of a single system.   

	We show that based on the level sets of the main eigenfunction of the Koopman operator we can extract information concerning the stable manifold of saddle points, and therefore, we can characterize the stability boundary of the attraction region of asymptotically stable equilibrium points.

	This paper is organized as follows. In Section~\ref{sec:LocDyn}, we present the analysis of a system based on MAK where the problem of complexity is presented by the order of the differential equation and the existence of multiple equilibria. In Section \ref{sec:ergodic_systems}, we show the current methods for finding the Koopman operator based on extended dynamic mode decomposition, an useful theorem for the characterization of attraction regions, and the way we use the operator to find these regions. In Section \ref{sec:application_example} we present our results for a particular problem within the MAK modeling strategy. A discussion which describes implications and advantages  of the analysis is presented in Section \ref{sec:analysis}. Finally, in Section \ref{sec:conclusions} we present some conclusions. 

	\textbf{Notation.} $\mathbb{C}$ denotes the field of complex numbers. $\mathbb{R}$ and $\mathbb{R}_+$ denote the field of real and nonnegative real numbers, respectively. For any vector $x \in \mathbb{R}^n$, $x^{\star}$ denotes  complex conjugate transpose, and $||x||$ represents the Euclidian norm. The space $n\times{}n$ of real matrices is designes by $R^{n\times{}n}$. For a complex number $\lambda$, $|\lambda|$ represents its norm. For any set $A$, $\bar{A}$ denotes the  closure of $A$.   Finally, the operator $\odot$ is defined as the product term to term. The vector exponentiation $M^{\pm{\eta}}$ is defined term by term.

\section{Local Dynamics}
	\label{sec:LocDyn}
	The distributed approach for treating the wastewater and flood control problem provides that the subsystems be represented in a suitable form. Biochemical dynamical systems can be represented by mass action kinetics (MAK)\cite{Chellaboina2009a,Feinberg1979c,Feinberg1979b,Feinberg1979a,Feinberg1979}, defined by autonomous nonlinear dynamical systems of the form
	\begin{equation}
		\dot{x}^{t}=f(x^{t},u^{t}),\,\,x(0)=x_0,\,\,t\in \mathbb{R}_{+},
		\label{eq:NormalDynSys}
	\end{equation}
	where $x^{t}\in\mathbb{R}^{n}$ is the state of the system, $u^{t}\in\mathbb{R}^{p}$ is the input or control signal of the system, and $f:\mathbb{R}^{n}\times\mathbb{R}^{p}\rightarrow\mathbb{R}^{n}$ is continuous in $x$ and $u$. Dynamical systems under the MAK structure produce differential equations in a polynomial form, as they represent concentrations, they cannot be negative, i.e., they live in the first orthant of the Euclidean space. 
	Starting from a set of species $\mathcal{S}=[s_1, s_2,\cdots, s_n]^{\top}\in\mathbb{R}^{n}$ with their respective concentrations $X=[x_1, x_2,\cdots, x_n]^{\top}\in{}\mathbb{R}^{n}$ such that $\sum_{j=1}^{n}x_j=1$, a set of reactions $\kappa{}=[k_1, k_2,\cdots, k_m]^{\top}\in{}\mathbb{R}^{m}$, and a stoichiometric network given by 
	\begin{equation}
		\sum_{j=1}^n\alpha_{i,j}x_j\overset{k_i}{\rightarrow{}}\sum_{j=1}^{n}\beta_{i,j}x_j.
		\label{eq:GenStoi}
	\end{equation}

	In matrix form, the reactants $\alpha_{i,j}$ and the products $\beta_{i,j}$ can be represented as $A=[\alpha_{i,j}]$ and $B=[\beta_{i,j}]$, $A,B\in\mathbb{R}^{m\times{}n}$ respectively. With this at hand, the differential equation governing the stoichiometric network \eqref{eq:GenStoi} is
	\begin{equation}
	  	\dot{x}=[B-A]^{\top}(\kappa{}\odot{}x^{A}),
		\label{eq:MAK}
	\end{equation}
	where the matrix exponentiation gives a vector $y^{A}\in\mathbb{R}^m$ such that $y_j=\prod_{k=1}^{m}x_{k}^{\alpha_{j,k}}$. Through the whole paper we are going to analize the next example based on a stoichiometric network.
	\begin{example}[label=ex1]
		Consider the simple network of an auto-catalytic replicator \cite{Gray1984} on a continuous flow stirred tank reactor (CFSTR) described by the network 
		\begin{equation}
\begin{tikzpicture}[->,>=stealth']
  \matrix (m) [matrix of math nodes, row sep=.3cm, column sep=.75cm, column 1/.style={nodes={align=right}}, column 2/.style={nodes={align=center}},column 3/.style={nodes={align=left}}]{$s_1$+$2s_2$&&$3s_2$\\$s_2$&&$s_3$\\$s_1$&$s_0$&$s_2$\\&&\\&$s_3$,&\\};
  \path
  ([yshift=.5mm]m-1-1.east) edge [above] node[yshift=-.9mm] {\scriptsize$k_{1}$} ([yshift=.5mm]m-1-3.west) 
  ([yshift=.5mm]m-2-1.east) edge [above] node[yshift=-.9mm] {\scriptsize$k_{2}$} ([yshift=.5mm]m-2-3.west) 
  ([yshift=.5mm]m-3-1.east) edge [above] node[yshift=-.9mm] {\scriptsize$g$} ([yshift=.5mm]m-3-2.west) 
  ([yshift=-.5mm]m-3-2.west) edge [below] node[yshift=.9mm] {\scriptsize$gx_{A}^{in}$} ([yshift=-.5mm]m-3-1.east) 
  ([yshift=-.5mm]m-3-2.east) edge [below] node[yshift=.9mm] {\scriptsize$gx_{B}^{in}$} ([yshift=-.5mm]m-3-3.west) 
  ([yshift=.5mm]m-3-3.west) edge [above] node[yshift=-.9mm] {\scriptsize$g$} ([yshift=.5mm]m-3-2.east) 
  ([xshift=-0.5mm]m-5-2.north) edge [left] node[xshift=.9mm,yshift=-.5mm] {\scriptsize$g$} ([xshift=-0.5mm]m-3-2.south) 
  ([xshift=0.5mm]m-3-2.south) edge [right] node[xshift=-.9mm] {\scriptsize$gx_{C}^{in}$} ([xshift=0.5mm]m-5-2.north); 0->D
\end{tikzpicture}

			\label{eq:AutocatR}
		\end{equation}
		where $s_1$ is the resource the species $s_2$ consumes to replicate, $s_3$ is the death species in the tank, $s_0$ is the environment, from which input reactants or species come from and go to at the input and output. $k_1>k_2$ are the replicator rate constant and the species death constant respectively. $g$ is the in/out-flow of the system, and $x^{in}_s\in\mathbb{R}^{m}$ are the concentration of reactants and product in the input feed. The differential equation for this system arises from:
		\begin{align}
			 A&=\begin{bmatrix}1 &2 &0 &0\\
		       0 &1 &0 &0\\
		       1 &0& 0 &0\\
		       0 &1 &0 &0\\
		       0 &0 &1 &0\\
		       0 &0 &0 &1\end{bmatrix};&B&=\begin{bmatrix}0 &3 &0 &0\\
		      0 &0 &1 &0\\
		      0 &0 &0 &1\\
		      0 &0& 0 &1\\
		      0 &0 &0 &1\\
		      1 &0 &0 &0\end{bmatrix};&\kappa&=\begin{bmatrix}k_1\\k_2\\g\\g\\g\\g\end{bmatrix},
		\end{align}
		where the effect of the input concentrations has been neglected because $x^{in}_{s}=[1\,0\,0]^{\top}$, which as a result gives the set of differential equations
		\begin{align}
			\dot{x}_1&=-k_{1}x_{1}{x_{2}}^2+g-gx_{1}\\ 
			\dot{x}_2&=+k_{1}x_{1}{x_{2}}^2-k_{2}x_{2}-gx_{2}\\ 
			\dot{x}_3&=+k_{2}\,x_{2}-gx_{3},
			\label{eq:SimpleDynamics}	
		\end{align}
		where the line concerning the dynamics of the $s_0$ species, which represents the exchange of material with the environment was discarded because it has no dynamics (i.e., $\dot{x}_0=0$). Note that any of the equations for the state variables can be discarded as it is a linear combination of the other two, then, the system can be treated as a two state variable problem (e.g., $x_1$ and $x_2$).
	\end{example}
	\begin{definition}
		 A vector $x^{*}\in{}\mathbb{R}^{m}$ satisfying $[B-A]^{\top}(\kappa{}\odot{}x^{A})=0$ is an equilibrium of \eqref{eq:MAK}.
	\end{definition}
	\begin{example}[continues=ex1]
		\begin{align} 
			x^{*}&=\left\{\vphantom{\frac{a}{a}}(x_1,x_2,x_3)\in{}\mathbb{R}^{3}_{+}:\begin{bmatrix}1&0&0\end{bmatrix}\right.,\\
			&\begin{bmatrix}1-\frac{\gamma+\sqrt{g}\sqrt{k_{1}}}{2\sqrt{g}\sqrt{k_{1}}}&\frac{\sqrt{g}\left(\gamma+\sqrt{g}\sqrt{k_{1}}\right)}{2\sqrt{k_{1}}\left(g+k_{2}\right)}&\frac{k_{2}\left(\gamma+\sqrt{g}\sqrt{k_{1}}\right)}{2\sqrt{g}\sqrt{k_{1}}\left(g+k_{2}\right)}\end{bmatrix},\label{eq:Equilibria}\\
			&\left.\begin{bmatrix}1+\frac{\gamma-\sqrt{g}\sqrt{k_{1}}}{2\sqrt{g}\sqrt{k_{1}}}&-\frac{\sqrt{g}\left(\gamma+\sqrt{g}\sqrt{k_{1}}\right)}{2\sqrt{k_{1}}\left(g+k_{2}\right)}&-\frac{k_{2}\left(\gamma+\sqrt{g}\sqrt{k_{1}}\right)}{2\sqrt{g}\sqrt{k_{1}}\left(g+k_{2}\right)}\end{bmatrix}\right\},
		\end{align}
		where $\gamma$ is defined as
		\begin{equation}
			\gamma(k_1,k_2,g)=\sqrt{-4g^2-8gk_{2}+k_{1}g-4{k_{2}}^2},
		\end{equation}
		and the equilibria is only defined for the interval in which $\gamma$ is real.
	\end{example}
	An analysis of the nature of the equilibria for the simple system based on classical techniques would be very difficult, the nature of the equilibria for every equilibrium depends on the values of $k_1$, $k_2$, and $g$.
	\begin{example}[continues=ex1]
		For the first point in \eqref{eq:Equilibria} a simple choice of Lyapunov function $V(x)$, see \cite{Khalil2015} for details on Lyapunov based stability,
		\begin{equation}
			V(x)=-\sum_{i=1}^{m}x^{*}_{i}\text{ln}\left(\frac{x_i}{x^{*}_{i}}\right),
		\end{equation}
		where $V(x^{*})=0$, and due to Jensen's inequality $\mathbb{E}f(x)\geq{}f(\mathbb{E}x)$,
		\begin{align}
			V(x)&=-\sum_{i=1}^{m}x^{*}_{i}\text{ln}\left(\frac{x_i}{x^{*}_{i}}\right)\\
				&\geq{}-\text{ln}\left(\sum_{i=1}^{m}x^{*}_{i}\frac{x_i}{x^{*}_{1}}\right)\\
				&=0.
		\end{align}
	Now, for the derivatives of $V(x)$ along the trajectories of $x$, and due to the fact that $x\in[0,1]$ we have
	\begin{align}
		\dot{V}(x)&=\begin{bmatrix}\frac{x^{*}_1}{x_{1}}&\frac{x^{*}_2}{x_{2}}&\frac{x^{*}_3}{x_{3}}\end{bmatrix}\dot{x}\\
				  &=-k_1x_2^2+g-\frac{g}{x_1}\\
				  &\leq{}0.
	\end{align}

	Note that $\dot{V}(x)\neq{}0$ for all $x\neq{}x_e$ and therefore the first equilibrium point is asymptotically stable, and, as $V(x)$ is not radially unbounded, the first equilibrium is not globally asymptotically stable. As $x_3$ can be expressed as a linear combination of $x_1$ and $x_2$, there is no uncertainty induced by the absence of $x_3$ in the derivative along the trajectories of the Lyapunov function.

	For the second and third equilibria in \eqref{eq:Equilibria}, the dynamic behavior is more complex, as they have a dependency on the parameters, and the second equilibrium produces a Hopf bifurcation \cite{Uppal1974,Uppal1976} that changes the nature of the equilibrium. The dynamic behavior of the equilibria with respect to different values of the in/out-flow constant $g$, and some fixed values of the reaction constants $k_1=10$ and $k_2=0.1$ is presented in Fig.~\ref{fig:EqDyn}.
	\begin{figure}[ht]
		\centering
		\includegraphics[width=0.45\textwidth]{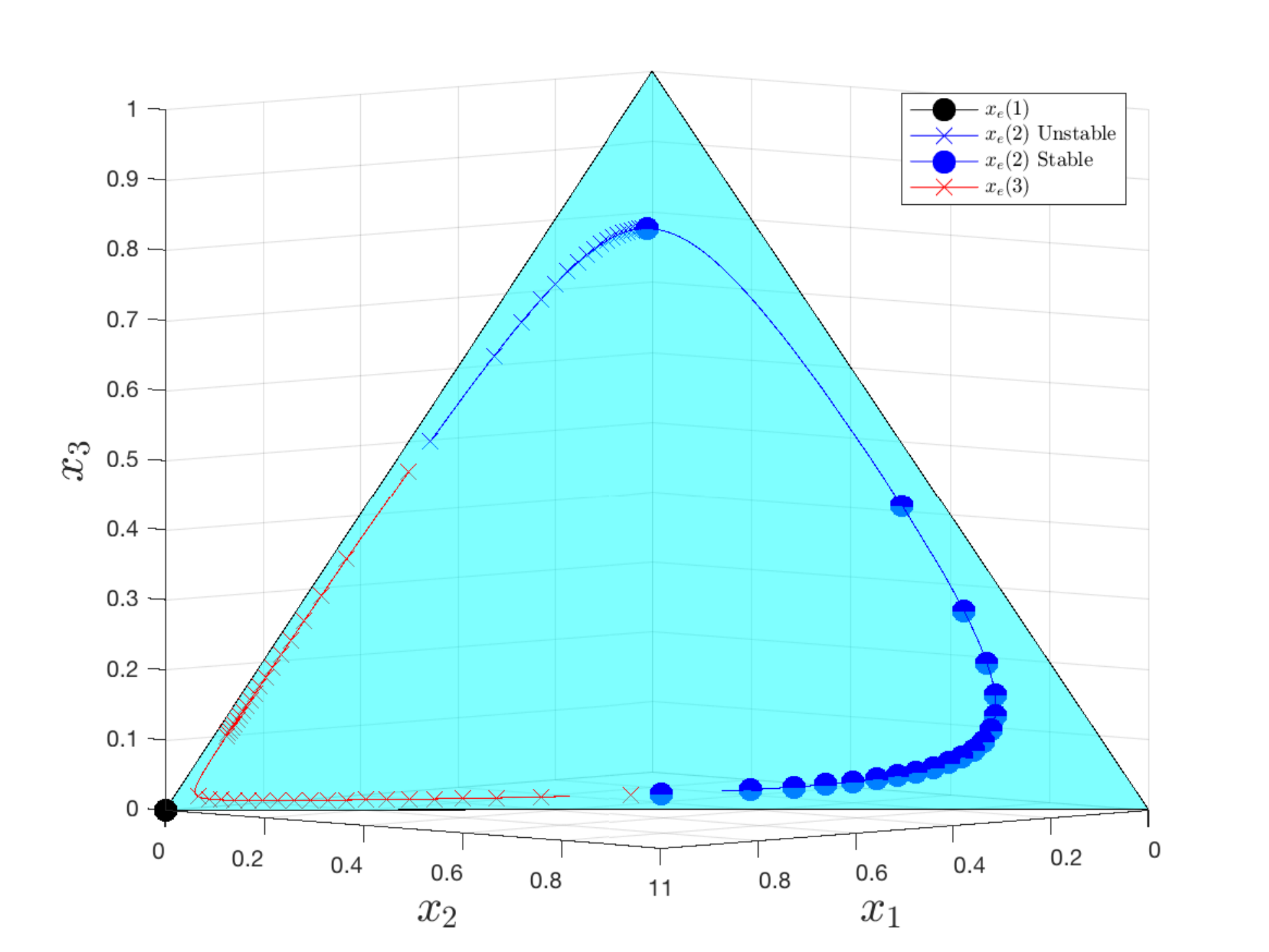}
		\caption{Equilibria, and its nature with respect to in/out-flow parameter $g$.}
		\label{fig:EqDyn}
	\end{figure}

	Although this process gives important information about of the local behavior of the different equilibria, a better analysis would consider regions of attraction of these equilibria for control and optimization of the process.
	\end{example} 
	Now, consider that the system belongs to a network of interacting subsystems that are set to collaborate into achieving a common goal while being controlled locally for their individual performance, such a subsystems interconnection would yield a structure such as
	\begin{equation}
		\dot{x}_i=f_{i}(x_i(t),u_i(t))+g(x(t)),\,\{i\}_{1}^{N_s},
		\label{eq:InterDynSys}
	\end{equation}
	where $f_i(x_i(t),u_i(t)):\mathbb{R}^{n_i}\times\mathbb{R}^{p_i}\rightarrow\mathbb{R}^{n_i}$ represents the local dynamics, and $g(x):\mathbb{R}^{\sum_{i}n_i}\rightarrow\mathbb{R}^{n_i}$ represents the effect due to interconnection, these for each of the $N_s$ subsystems. When considering interconnection the dimension of the dynamical system grows. Depending on how large is this growth,  on the order of the individual systems, the number of individual systems and the interconnection structure, the simple task of finding the equilibria may not be possible analytically, and numerical methods may not give solutions within acceptable computing times \cite{Shmakov2011}.

	The proposed approach combines two techniques for the solution of the equilibria analysis of the interconnected system via numerical methods. These present two advantages with respect to the classical approach. The first is in terms of dimensionality, and the second is in terms of dependence on the model. The operator based technique, such as the Koopman operator, relies on a data set of snapshot pairs of the system, e.g., the full state observable of the state space with a specific sampling time \cite{Williams2016,Williams2015}. This approach allows us to complement algorithmic techniques such as \cite{chiang2015stability}, to characterize nonlinear dynamical systems. These techniques depend on the model and the need of backward integration to better characterize the system. In the next section the \textit{extended dynamic mode decomposition} (EDMD) and the \textit{Stability Regions of Nonlinear Dynamical Systems} (SR) are going to be explained as well as the way in which these two can be combined and complemented.  


\section{The Koopman Operator and Attraction Regions} 
	\label{sec:ergodic_systems}
	The approach for the analysis of the dynamical characteristics of the systems, as a single entity and in interconnection is based on EDMD and SR. The following algorithm for EDMD is from \cite{Williams2015}, and references therein where the theory behind the algorithm is presented.
	\subsection{Extended Dynamic Mode Decomposition} 
		\label{sub:EDMD}
		Consider the discrete time dynamical system $(\mathcal{M}_{d};\eta;F_{d}^{\eta})$, where $\mathcal{M}_{d}\subseteq{}\mathbb{R}^N$ is the state space, $\eta\in\mathbb{Z}$ is the discrete time, anf $F_{d}:\mathcal{M}_{d}\rightarrow\mathcal{M}_{d}$ is the evolution operator. The Koopman $\mathcal{K}$ operator acts on functions of the state space $\psi\in\mathcal{F_{d}}$, where $\psi:\mathcal{M}:\rightarrow\mathbb{C}$, i.e., defines the evolution of the observable $\psi$ governed by the linear operator $\mathcal{K}$, defining a new dynamical system $(\mathcal{F};n;\mathcal{K})$. For the case of autonomous continuous time dynamical systems such as \eqref{eq:NormalDynSys}, there is an associated continuous flow map $F_{c}^{t}(x):\mathcal{M}\rightarrow\mathcal{M}$ (i.e., the solution of the differential equation for $t\in\bar{\mathbb{R}}_{+}$) that defines the continuous time dynamical system $(\mathcal{M};F_{c}^{t})$. For this dynamical system, we can define the Koopman operator $\mathcal{K}_{\Delta{}t}$, associated with the flow of the system in a fixed time interval $\Delta{}t$, defining again a new dynamical system $(\mathcal{F};\eta;\mathcal{K}_{\Delta{}t}^{\eta})$, i.e., a dynamical system defined by its discrete time evolution operator $\mathcal{K}_{\Delta{}t}$. 

		The Koopman operator based algorithm requires $P$ pairs of data snapshots, either from a real system or a computationally integrated one at a specific sampling $\Delta{}t$. The snapshot pairs, $\{(x_p,y_p)\}_{p=1}^{P}$ are organized in data sets
		\begin{equation}
			X=\begin{bmatrix}x_1&x_2&\ldots{}&x_P\end{bmatrix},\,\,\,Y=\begin{bmatrix}y_1&y_2&\ldots{}&y_P\end{bmatrix},
		\end{equation}
		where $x_i,y_i\in{}\mathcal{M}$, the space state of the system, and a dictionary of observables $\mathcal{D}=\{\psi_1,\ldots{}\psi_{N_k}\}$, where $\psi_{i}\in\mathcal{F}$ are the observables that define the vector valued function $\Psi{}:\mathcal{M}\rightarrow\mathbb{C}^{1\times{}N_k}$ where
		\begin{equation}
			\Psi{}(x)=\begin{bmatrix}\psi_1(x)&\psi_2(x)&\ldots{}&\psi_{N_k}(x)\end{bmatrix}.
		\end{equation}
		
		The goal now is to generate the finite-dimensional approximation of the infinite linear operator $\mathcal{K}$, defined as $K\in\mathbb{C}^{N_k\times{}N_k}$ and calculated 
		\begin{equation}
			K\triangleq{}G_k^{+}A_k,
		\end{equation}
		where $+$ denotes pseudoinverse, and matrices $G_k,A_k\in\mathbb{C}^{N_k\times{}N_k}$ are defined by
		\begin{align}
			G_k=&\frac{1}{P}\sum_{p=1}^{P}\Psi(x_p)^{\star}\Psi(x_p),\\
			A_k=&\frac{1}{P}\sum_{p=1}^{P}\Psi(x_p)^{\star}\Psi(y_p),
		\end{align}
		where $\star$ denotes complex conjugate. With the $K$ matrix, its eigenvalues $M_k=\text{diag}([\mu_1,\ldots,\mu_{N_k}])$, and right eigenvectors $\Xi=[\xi_1,\ldots,\xi_{N_k}]$ can be calculated. Koopman eigenfunctions can be defined as
		\begin{equation}
			\Phi=\Psi(x)\Xi,
		\end{equation}
		with $\Xi$ as the matrix of right eigenvectors, its inverse is the matrix of left eigenvectors, $\Xi^{-1}=W_k^{\star}$. Now, defining the full-state observable, $g_k(x)=x$ such that $g_k(x):\mathcal{M}\rightarrow\mathbb{R}^N$, where $g_k(x)\in\mathcal{F_\mathcal{D}}\subseteq\mathcal{F}$, i.e., the full state observable lies in the space spanned by $\mathcal{D}$, which is a subspace of $\mathcal{F}$. With this result, the full state observable can be expressed as
		\begin{equation}
		 	g_k(x)=(\Psi(x)B_k)^{\top},
		\end{equation} 
		 where $B_k\in\mathbb{R}^{N\times{}N_K}$ is an appropriate matrix of weights that relate the leading Koopman eigenfunctions to each of the full state observables values for them to accurately capture the state. The literature lacks a proper way of characterizing these vectors that compose the matrix $B_k$. With the appropriate weights for each of the leading eigenfunctions, the full state observable can be calculated as
		 \begin{equation}
		 	g_k(x)=V_k\Phi(x)^{\top},\qquad{}V_k=(W_k^{\star}B_k)^{\top},
		 \end{equation}
		 where the $i$th column of $V_k$ is the $i$th Koopman mode, leaving us with an appropriate definition of the full state observable in terms of Koopman eigenfunctions and modes of the Koopman operator $K$. This definition serves to define the evolution of the states in terms of the Koopman characterization $\{(\mu_k,\phi_k,v_k)\}_{k=1}^{N_k}$ consisting of $N_k$ Koopman eigenvalues, eigenfunctions, and modes. Therefore, the $\eta^{\text{th}}$ application of the state evolution $F^{\eta}(x)=x_{\eta}$ from the initial condition $x_{0}$ is
		 \begin{equation}
		 	x(\eta)=F^{\eta}(x_{0})=K^{\eta}g_k(x_{0})=(M_k^{\eta}\odot{}V_k)\Phi(x_{0})^{\top}.
		 	\label{eq:ForwardFlow}
		 \end{equation}
		 
		 Note that the Koopman representation of the system also allows for a clean representation of the backwards evolution of the states given by
		 \begin{equation}
		 	x(-\eta)=\left(M_k^{-\eta}\odot{}V_k\right)\Phi(x)^{\top}.
		 	\label{eq:BackwardFlow}
		 \end{equation}
		 \begin{remark}
		 	The Koopman eigenfunctions $\Phi(x)$ contain information about the equilibria of the system in terms of where they are, and the functions reveal the stable and unstable manifold of saddle point equilibria.	  With the deduction of the Koopman eigenfunctions and the evolution of the system, the stability boundary of asymptotically stable equilibria can be characterized in the following way.
		 \end{remark}
		
	\subsection{Stability Boundary} 
		 \label{sub:stability_boundary}
		 The theory for stability boundary for systems that contain saddle points that characterize the attraction region of an asymptotically stable equilibrium point was developed by \cite{chiang2015stability}. Equilibria is divided between asymptotically stable equilibria $x_{s}$, and saddle/hyperbolic equilibria $\hat{x}$. For each of them we can define the attraction region of the asymptotically stable equilibria defined as $A_{s}(x_{s})\triangleq\{x\in\mathbb{R}^n:\lim_{t\rightarrow{\infty}}F_{c}^{t}(x)=x_{s}\}$, the boundary of the attraction region as $\partial{}A_{s}(x_{s})=\overline{A}_{s}(x_{s})\cap{}(\overline{\mathbb{R}^{n}\backslash{}A_{s}(x_{s}))}$, the stable manifold of the saddle point $W^{s}(\hat{x})=\{x\in\mathbb{R}^{n}:F_{c}^{t}(x)\rightarrow{}\hat{x}\text{ as }t\rightarrow{}\infty{}\}$, and the unstable manifold of the saddle point $W^{u}(\hat{x})=\{x\in\mathbb{R}^{n}:F_{c}^{t}(x)\rightarrow{}\hat{x}\text{ as }t\rightarrow{}-\infty\}$ (assuming there is an inverse for $F_c^{t}$ for the backward flow of the system). With these definitions, we can  present the main theorem for the characterization of attraction regions {\cite[Th.~4-8]{chiang2015stability}}.
		 \begin{theorem}
		 	For a nonlinear autonomous dynamical systems \eqref{eq:NormalDynSys} which satisfies the following assumptions.
		 	\begin{enumerate}
		 		\item All equilibria on the stability boundary are saddle points.
		 		\item The stable and unstable manifolds of the equilibria on the stability boundary satisfy the transversality condition.
		 		\item Every trajectory on th stability boundary approaches one of the saddle point equilibria as $t\rightarrow\infty$.
		 	\end{enumerate}
		 	Let $\{\hat{x}_i\}_{i\in{}\mathbb{N}}$ be the saddle equilibria on the stability boundary $\partial{}A_s(x_{s})$ of an asymptotically stable equilibrium point $x_{s}$. Then
		 	\begin{enumerate}[(a)]
		 		\item $\hat{x}_{i}\in\partial{}A_s(x_{s})$ if and only if $W^{u}(\hat{x}_i)\cap{}A_s(x_{s})\neq\varnothing$
		 		\item $\partial{}A_s(x_{s})=\cup{}W^{s}(\hat{x}_i)$.
		 	\end{enumerate}
		 \end{theorem}
		 Next, we are going to to show how the eigenfunctions from EDMD can be used to characterize the attraction region of an asymptotically stable equilibrium point.
	\subsection{Koopman Operator for Attraction Regions} 
		\label{sub:koopan_for_attraction_regions}
		The characterization of the attraction region of an asymptotically stable equilibrium point requires that the stable manifolds of saddle points on the boundary of attraction be characterized. For this purpose, we propose the use of some important results coming from the operator representation that EDMD gives. First, we are going through some description and assumptions of the kind of systems being evaluated.

		The systems under evaluation are under the MAK modeling strategy, these have a nice polynomial structure coming from the stoichiometric network of the system. Although we know the polynomial structure arising from MAK, we assume that the reaction constants for the process are not known, but that they are constant nonetheless, i.e., they are independent of the reactor temperature, either because the system is adiabatic or, from the assumption that a proper controller has been setup for this purpose. This gives us some notion of the characteristics of the equilibria for the system, but not the whole picture. 

		We also assume that there is information concerning several trajectories of the system, i.e., the system has been tested from several initial conditions such that a set of snapshots can be used to characterize the whole dynamic spectra, which is possible due to the fact that MAK lie in the unitary manifold of the first orthant. These snapshots correspond to a full measurement of the state at a specified sampling time. Now, the procedure is to identify the attraction regions, characterized by the stability boundary, defined by the stable manifold of saddle points in the boundary. 

		Consider the autonomous nonlinear system \eqref{eq:NormalDynSys} whose dynamics come from \eqref{eq:MAK}, based on a stoichiometric network such as \eqref{eq:AutocatR}. Recall from the EDMD procedure that the forward and backward flow of the system can be achieved in terms of the eigenvalues, eigenfunctions, and Koopman modes, \eqref{eq:ForwardFlow} and \eqref{eq:BackwardFlow}. Therefore if any of these trajectories is invariant with respect to the flow map or evolution operator of the system $F^{\eta}(x_0)$, then the system is in a fixed point.
		\begin{definition}[Fixed Point]
			\label{def:FixedPoint}
			Let $x^{*}$ be a fixed point of the dynamical system $(\mathcal{M};F_c^{t}(x))$, then the flow map $F_{c}^{t}(x^{*})$ maps into itself. i.e.,
			\begin{equation}
				F_{c}^{t}(x^{*})=x^{*}.
			\end{equation}
		\end{definition}

		As it is the case, we propose minimization of the squared norm of the difference between $x$ and the flow map of the system over the state space in order to find the equilibria.
		\begin{proposition}[Fixed points] 
			Let $(\mathcal{M};F_c^{t}(x))$ be a dynamical system that accepts an operator representation $(\mathcal{F};\eta;\mathcal{K}_{\Delta{}t}^{\eta})$ based on EDMD, then its evolution operator is defined by \eqref{eq:ForwardFlow}, and its equilibria can be found by solving
			\begin{equation}
				x^{*}=\min_{x}\left\Vert(M_k^{\eta}\odot{}V_k)\Phi(x)^{\top}-x\right\Vert,
			\end{equation}
			where all the local minima represent a fixed point of the system.
		\end{proposition}
		\begin{proof}
			This is a consequence of Def~\ref{def:FixedPoint}.
		\end{proof}
		\begin{remark}
			Note that this procedure is possible under the MAK paradigm as the system is defined on the unitary manifold of the first orthant, and therefore a complete exploration of the space state from several initial conditions is possible to find all the equilibria of the system.
		\end{remark}
		\begin{remark}
			This proposition is only accountable for the location of the equilibria in the space state, it does not give information about its stability characteristic.
		\end{remark}
		With the location of the equilibria, the known structure of the dynamical system based on the stoichiometric network comes into place in order to know the proper way to analyze the system. There are techniques that consider the reversibility of the stoichiometric network, and the deficiency of the system in order to draw conclusions about the number of equilibria and their behavior \cite{Feinberg1979c,Feinberg1979b,Feinberg1979a,Feinberg1979,Feinberg1980,Feinberg1987,Chellaboina2009a,Haddad2010}. We are focusing our efforts for the case where a saddle equilibrium can characterize the attraction region of an asymptotically stable equilibrium point via the stable manifold of a saddle equilibrium, therefore, we assume that by linearization, or by using the structure of the system and its properties, the saddle point is identifiable. With this information, the next proposition is the main result of this work, and is the fact that the stable manifold of a saddle equilibrium has a level set on the main eigenfunction describing the system.
		\begin{proposition}
			\label{prop:Manifold}
			Let $(\mathcal{M};F_c^{t}(x))$ be a dynamical system that accepts an operator representation $(\mathcal{F};n\Delta{}t;\mathcal{K}_{\Delta{}t})$ based on EDMD, then its main eigenfunction $\phi_{+}(x)$ corresponding to the eigenvalue with the largest norm is constant along the trajectory of the stable manifold of the saddle equilibria. Therefore, the stable manifold is defined by the level set 
			\begin{equation}
				L_{W^{s}(\hat{x})}(\phi_{+}(x))=\{(x_1,\dots,x_n):\phi_{+}(x_1\dots,x_n)=\phi_{+}(\hat{x})\}
			\end{equation}
		\end{proposition}
		\begin{proof}
			This is a concequence from the fact that $K^{t}\phi_{i}=\lambda_{i}\phi_{i}$, and the linearity of the operator \cite{Koopman1931,Mauroy2016}.
		\end{proof}

\section{Application Example} 
	\label{sec:application_example}
	Consider the application example of, the auto-catalytic simple system \eqref{eq:AutocatR}, assume that the structure is known, and that we are working in the region where the in/out-flow $g$ induces a stable focus on the third equilibrium point.
	\begin{example}[continues=ex1]
		We run the system 100 times with a length of 21 points per run with a $\Delta{}t=0.125$, that gives us a total of 2,000 snapshots. With that information, we apply EDMD from Section \ref{sub:EDMD} with a choice of dictionary composed by 25 entries of Hermite Polynomials. The system under study is the reduced dynamics due to the dependence of $x_3$ on the other two, that is, \eqref{eq:SimpleDynamics} without the last equation. 
		The result is an EDMD based Koopman operator that effectively represents the dynamics of the system with 25 eigenfunctions $\phi(x)$ that characterize the dynamical structure of the system. The first and second eigenfunctions have the associated eigenvalue with the largest norm, and their evaluation on the unitary manifold that contains the dynamics can be seen in Fig.~\ref{fig:SurfMani}, where the level set, obtained by evaluating the main eigenfunction as in Proposition~\ref{prop:Manifold} is also shown. 		
		\begin{figure}[ht]
		 	\centering
		 	\includegraphics[width=0.5\textwidth]{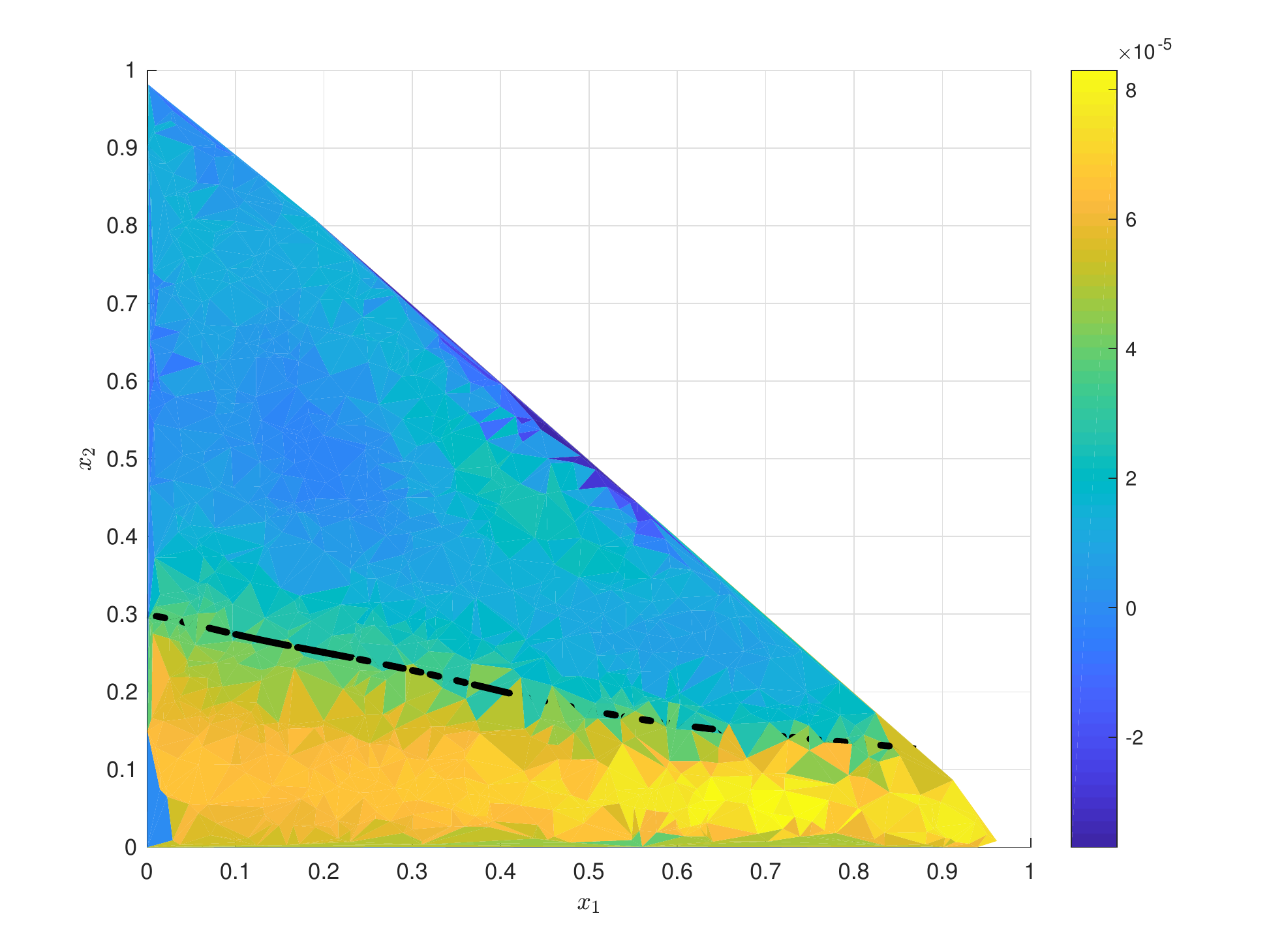}
		 	\caption{$\phi_{+}(x)$ and the level set $L_{W^{s}(\hat{x})}(\phi_{+}(x))$ that characterizes the stable manifold of the saddle point.}
		 	\label{fig:SurfMani}
		\end{figure} 
		The result of this approach is presented in Fig.~\ref{fig:DuffEqEvol}, where two data sets of 70 initial conditions per region are generated according to their location with respect to the level set that characterizes the attraction region of the asymptotically stable equilibria. In solid circles are the initial conditions that satisfy
		\begin{align}
		 	IC_{\bullet}=&\{(x_1,x_2)\in\mathbb{R}^{2}:\,x_1+x_2\leq1\}\cup\\
		 					&\{(x_1,x_2)\in\mathbb{R}^{2}:(x_1,x_2)\geq{}L_{W^{s}(\hat{x})}(\phi_{+}(x))\},
		\end{align} 
		and in solid diamonds the initial conditions that satisfy
		\begin{align}
			IC_{\blacklozenge}=&\{(x_1,x_2)\in\mathbb{R}^{2}:\,x_1+x_2\leq1\}\cup\\
		 					&\{(x_1,x_2)\in\mathbb{R}^{2}:(x_1,x_2)<L_{W^{s}(\hat{x})}(\phi_{+}(x))\}.
		\end{align}
		
		Taking trajectories from these distinct initial conditions gives the result as to which of the equilibria they converge, where it can be seen that there is some error due to the numeric nature of the algorithm, despite the error, a measure of the minimal distance between the focus equilibrium and the level set gives a measure on robustness of that equilibria with respect to any possible perturbations, i.e., $0.5$ for the case under consideration.
		\begin{figure}[ht]
		 	\centering
		 	\includegraphics[width=0.5\textwidth]{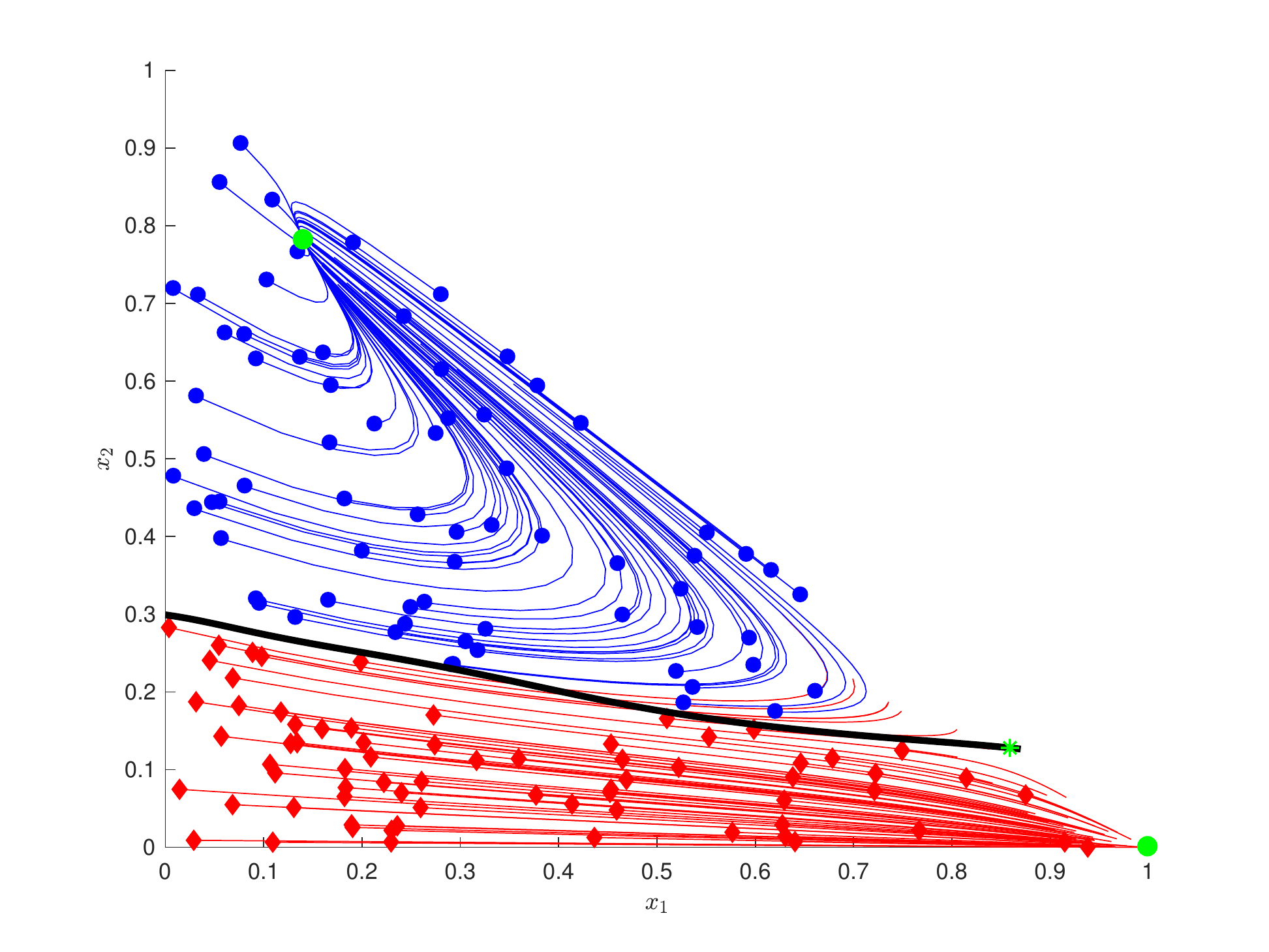}
		 	\caption{Evolution of trajectories with initial conditions based on the stable manifold of the saddle criteria.}
		 	\label{fig:DuffEqEvol}
		\end{figure} 
	\end{example}
\section{Discussion} 
	\label{sec:analysis}
	Our proposed approach relies on the proper analysis of eigenfunctions. This paper only covers the analysis of the level sets of eigenfunctions. This analysis lets us characterize the stable manifold of saddle points to obtain the attraction region of an asymptotically stable equilibrium. This is among the tools to extract information out of the eigenfunctions \cite{Mauroy2013,Mauroy2016,Mezic2013} to properly analyze a dynamical system. These numeric tools represent a base for the future analysis of complex systems in order to synthesize controllers for polynomial dynamical systems (e.g., MAK). Although there are other ways to analyze systems in terms of their attraction regions, we only need the structure of the system in order to perform the analysis. A qualitative representation is enough.   

	Even if we had the tools for extracting all kinds of information regarding the system, there are limitations with respect to the deduction of the eigenfunctions. The algorithms are not clear on the appropriate number of samples or the distribution of initial conditions in order to have an accurate representation of the system. There is no clear criteria to select a dictionary and its size. This produces inaccuracies in the eigenfunction, and as can be seen in Fig.~\ref{fig:DuffEqEvol}, some trajectories with initial conditions in the estimated region for the trivial equilibrium converge to the stable focus. 

	A possible improvement would be with respect to the choice of the main eigenfunction of the operator. All eigenfunctions capture different information and properties (relevant or not, accurate or not). A linear combination of the different level sets that show a similar behavior could yield better results. Unfortunately, the shortcoming stated earlier and the fact that the choice of dictionary affects the information that a particular eigenfunction captures, there is no proper way of selecting the main eigenfunctions.

	These data-driven techniques allow us to work over broader spectra of the state space; when dealing with systems where the exact dynamics are unknown, in presence of uncertainties in the system parameters or in the presence of perturbation that can get the system out of its operation point or even destabilize it. Therefore, these data-driven approach can complement classical control synthesis techniques to get a handle on optimality and robustness.

	To improve on the analysis, we need to explore the polynomial structure of MAK and find the relationship with the polynomials of the dictionary. This will allow us to construct dictionaries which will provide eigenfunctions with more information regarding the characterization of the system. This way, reduced order dictionaries could be employed to have a better handle on complexity.
\section{Conclusions} 
	\label{sec:conclusions}
	In conclusion, eigenfunctions coming from operator based algorithms provide information for the system analysis; in this particular case, the equilibria of the system and the attraction regions of asymptotically stable equilibria. The problem relies on the proper construction of eigenfunctions and the extraction of the information they provide.

	The accuracy of the estimation of the stability boundary depends on the number of available snapshots, the kind of dictionary used to obtain the Koopman operator matrix, and the size of the dictionary.

	We can conclude that the operator based analysis of dynamical systems is far from complete; we need a stronger theoretical framework in order to analyze systems with inputs, interconnection and feedback controllers. 

\bibliographystyle{ieeetr}
\bibliography{library.bib}
\end{document}